\documentclass{amsart}
\usepackage{xcolor}

\usepackage{amsthm,amsmath,amssymb,amsfonts,bm}
\usepackage{stmaryrd}
\usepackage{bbold}
\usepackage{times}
\usepackage[all]{xy}
\usepackage{mathrsfs}
\usepackage{graphicx}
\usepackage{tikz}
\usetikzlibrary{matrix,calc,arrows}
\usepackage{comment}

\theoremstyle{plain}
\newtheorem{thm}{Theorem}
\newtheorem{lem}[thm]{Lemma}

\newtheorem{prop}[thm]{Proposition}

\theoremstyle{definition}

\newtheorem{example}[thm]{Example}

\newtheorem*{notat}{Notations and conventions}

\theoremstyle{remark}

\newcommand{\End}{\operatorname{End}}

\newcommand{\ann}{\operatorname{ann}}

\newcommand{\acts}{%
  \!\mathrel{\begin{tikzpicture}[scale=.8,baseline=(current  bounding  box.south)] 
  \useasboundingbox (-.6,-.2) rectangle (.1,.2);
  \node at (0,-.06) {} edge[out=210,in=150,loop] ();
\end{tikzpicture}}\!\!\!\!
}

\newcommand{\Th}{^\text{\rm th}}

\newcommand{\rad}{\operatorname{rad}}

\newcommand{\diag}{\operatorname{diag}}

\newcommand{\ch}{\operatorname{char}}

\newcommand{\into}{\hookrightarrow}

\newcommand{\Mat}{\operatorname{Mat}}

\DeclareMathOperator{\GKdim}{GKdim}

\renewcommand{\k}{\mathbb{k}}
\renewcommand{\phi}{\varphi}

\newcommand{\bdot}{\,\text{\raisebox{-.45ex}{$\boldsymbol{\cdot}$}}\,}

\newcommand{\sE}{\mathscr{E}}
\newcommand{\sC}{\mathscr{C}}
\newcommand{\sD}{\mathscr{D}}

\newcommand{\sG}{\mathscr{G}}
\newcommand{\sS}{\mathscr{S}}
\newcommand{\sT}{\mathscr{T}}
\newcommand{\sN}{\mathscr{N}}
\newcommand{\sM}{\mathscr{M}}
\newcommand{\sR}{\mathscr{R}}

\newcommand{\sZ}{\mathscr{Z}}

\newcommand{\sX}{\mathscr{X}}

\newcommand{\cS}{\mathcal{S}}

\newcommand{\Sy}{\mathcal{S}}

\renewcommand{\d}{\delta}

\newcommand{\cen}{\mathcal{Z}}

\newdir{ >}{{}*!/-5pt/\dir{>}}

\begin{document}

\title[Affine representable algebras]%
{A note on affine representable algebras}

\author{Martin Lorenz}
\address{Department of Mathematics, Temple University, Philadelphia, PA 19122, USA}

\subjclass[2020]{16P90, 16S20, 16S50}

\keywords{growth equivalence of algebras, Gelfand-Kirillov dimension, 
affine algebras, representable algebras,
virtually commutative algebras, PI algebras, subcentralizing extensions}

\begin{abstract}
We consider affine representable algebras, that is, finitely generated algebras over a field $\k$
that can be embedded into some matrix algebra over a commutative $\k$-algebra $C$. 
We show that $C$ can in fact be chosen to be a polynomial algebra over $\k$.
We also give a refined version of a theorem of V.T. Markov stating that
the Gelfand-Kirillov dimension of any affine representable algebra is an integer.
\end{abstract}

\maketitle

%%%%%%%%%%%%%%%%%%%%%%%%%%%%%%%%%%%%%%%%%%%%%%%%%%
% Intro
%%%%%%%%%%%%%%%%%%%%%%%%%%%%%%%%%%%%%%%%%%%%%%%%%%

\section{Introduction}

%%%%%%%%%%%%%%%%%%%%%%%%%%%%%%%%%%%%%%%%%%%%%%%%%%

\subsection{Some classes of noncommutative algebras}
\label{SS:IntroClasses}

Let $R$ be an algebra over a field $\k$, associative with $1$ 
but generally noncommutative. The algebra $R$ is called \emph{affine} if
there is a finite set of algebra generators and \emph{represen\-table} if $R$ embeds into a matrix algebra
over some commutative $\k$-algebra.
The class of representable algebras, introduced by
Malcev \cite{aM43},%
\footnote{Malcev \cite{aM43} requires
an embedding into matrices over some $\k$-field. 
For affine $\k$-algebras, the two concepts are equivalent \cite{kB86}; 
see also Proposition~\ref{P:Coefficient}.}
consists of \emph{PI algebras}, that is, algebras satisfying a 
polynomial identity, and it includes all affine algebras that are
finite (i.e., finitely generated as module) over their centers \cite{kB86}. Thus, referring to algebras 
finite over their centers as \emph{virtually commutative},
we have the following hierarchy of algebras:
\begin{center}
PI $\supset$ representable $\supset$ affine virtually commutative.
\end{center}
Affine virtually commutative algebras are
noetherian by Hilbert's Basis Theorem and the Artin-Tate Lemma \cite[Exercise 1.1.8]{mL18} 
and they are (strongly) finitely presented \cite{mL88a}, \cite{mL88}. Affine
representable algebras, on the other hand, need not have either of these properties.
There are also many affine PI algebras that
are not representable; see \S\ref{SS:GK} below. 
Our focus will be on affine representable algebras, but PI algebras and virtually commutative algebras 
will also play a role. See \cite{lRlS15}
for background on general representable algebras.

%%%%%%%%%%%%%%%%%%%%%%%%%%%%%%%%%%%%%%%%%%%%%%%%%%

\subsection{Coefficient algebras}
\label{SS:Coefficient}

Let $R$ be affine representable, say
$R \into \Mat_d(C)$ for some commutative $\k$-algebra $C$. We will call $C$ 
a \emph{coefficient algebra} for $R$. 
Of course, any commutative $\k$-algebra containing $C$ also serves as a
coefficient algebra. On the other hand, since $R$ is assumed affine, every coefficient
algebra certainly contains one that is affine. In fact, more specific choices can be made.
The following proposition may be known---the proof uses standard arguments---but I have 
not been able to locate a reference.

\begin{prop}
\label{P:Coefficient}
Every affine representable $\k$-algebra has a coefficient algebra that is a
polynomial algebra over $\k$.
\end{prop}

%%%%%%%%%%%%%%%%%%%%%%%%%%%%%%%%%%%%%%%%%%%%%%%%%%

\subsection{Gelfand-Kirillov dimension}
\label{SS:GK}

Let $S$ be an arbitrary $\k$-algebra.
For any finite subset $\sS \subseteq S$ and any integer $n \ge 0$, 
let $\sS^{(n)}$ denote the $\k$-subspace of 
$S$ that is generated
by all products of the form $s_1s_2\cdots s_m$ with $s_i \in \sS$ and $m \le n$. So 
$\sS^{(n)} \subseteq \sS^{(n+1)}$ and $\sS^{(n)}\sS^{(n')}  = \sS^{(n+n')}$. 
The \emph{Gelfand-Kirillov dimension}  of $S$ is defined by
\[
\GKdim S = \sup_\sS \varlimsup_n \log_n \dim_\k\sS^{(n)}, 
\]
where $\log_n \bdot = \log \bdot /\log n$ is the base-$n$ logarithm \cite{iGaK66}.
If $T$ is another $\k$-algebra, then we write
\begin{equation}
\label{E:Relation}
S \preceq T \quad \underset{\text{def}}{\iff}\quad
 \begin{minipage}{3.6in}
For every finite subset $\sS \subseteq S$, there is 
a finite subset $\sT \subseteq T$ and a constant $K$ 
such that $\dim_\k \sS^{(n)} \leq K \dim_\k \sT^{(n)}$
for all $n$.
\end{minipage}
\end{equation}
The relation $\preceq$ is evidently transitive; $S \preceq T$ holds whenever $S$ is a subalgebra 
or a homomorphic image of $T$; and
\begin{equation}
\label{E:GKequiv}
S \preceq T \implies \GKdim S \le \GKdim T.
\end{equation}
If $S \preceq T$ and $T \preceq S$, then we write $S \equiv T$ and call the algebras
$S$ and $T$ \emph{growth equivalent}. In this case, $\GKdim S = \GKdim T$ by \eqref{E:GKequiv};
in general, however, growth equivalence is finer than equality of 
Gelfand-Kirillov dimension \cite[\S1-2]{wBhK76}.
\smallskip

Our second contribution refines an earlier
result due to V.T. Markov \cite{vMxx}. 

\begin{thm}
\label{T:Markov}
Let $R$ be an affine representable $\k$-algebra. 
Then there is an affine commutative $\k$-algebra $D$
such that $R  \equiv D$. In particular, $\GKdim R$ is an integer.
\end{thm}

The second assertion of the theorem follows from the first,
since the Gelfand-Kirillov dimension of any affine
commutative algebra is an integer \cite[3.2]{wBhK76}.
In contrast, the Gelfand-Kirillov dimension of an affine PI algebra can be $0, 1$ or
any real value $\ge 2$ \cite[2.11]{wBhK76}. 
Integrality of $\GKdim R$ was stated with a brief indication of a proof in the
unpublished typescript \cite{vMxx}, which 
was made available to me by Lance Small around 1991.  I worked out the details
of Markov's theorem and my notes were eventually published as
\cite[Section 12.10]{gKtL00}. The proof of Theorem~\ref{T:Markov} given below follows 
the same general outline.

%%%%%%%%%%%%%%%%%%%%%%%%%%%%%%%%%%%%%%%%%%%%%%%%%%

\begin{notat}
Algebras are assumed associative and unital. Centers will be denoted by $\cen \bdot$\,.
The notation and terminology used above will be retained throughout. 
In particular, $R$ will always denote an affine algebra over an arbitrary base field $\k$.
\end{notat}

%%%%%%%%%%%%%%%%%%%%%%%%%%%%%%%%%%%%%%%%%%%%%%%%%%

\section{Some preparations} 
\label{S:Prep}

%%%%%%%%%%%%%%%%%%%%%%%%%%%%%%%%%%%%%%%%%%%%%%%%%%

\subsection{Growth equivalence}
\label{SS:Relation}

It suffices to satisfy the condition in \eqref{E:Relation} for all $n \gg 0$,
since $K$ may be adjusted to include finitely many initial values of $n$. 
It is also enough to show $\dim_\k \sS^{(n)} \leq K \dim_\k \sT^{(cn)}$
for some constant $c$, because this implies $\dim_\k \sS^{(n)} \leq K \dim_\k \sT'^{(n)}$
for any finite generating set $\sT'$ of $\sT^{(c)}$. In particular,
if $S$ is affine, then it suffices to ensure that a fixed 
finite set of algebra generators $\sG \subseteq S$ 
satisfies the condition in \eqref{E:Relation}: An arbitrary finite
subset $\sS \subseteq S$ is contained in $\sG^{(c)}$ for some $c$; so
$\dim_\k \sG^{(n)} \le K \dim_\k \sT^{(n)}$ for all $n$ implies
$\dim_\k \sS^{(n)} \le \dim_\k \sG^{(cn)} \le K \dim_\k \sT^{(cn)}$.
\smallskip

As for growth equivalence, note that $S \equiv \k$ if and only if every affine
subalgebra of $S$ is finite dimensional, which in turn is equivalent to
$\GKdim S = 0$. Furthermore, by Noether normalization and part (b)
of the lemma below, any affine commutative algebra $S$ is growth equivalent to
a polynomial algebra in $\GKdim S$ many variables over $\k$. Thus, the converse of 
\eqref{E:GKequiv} holds for affine commutative algebras $S$ and $T$.

\begin{lem}
\label{L:Relation1}
Let $S$ and $T$ be $\k$-algebras.
\begin{enumerate}
\item
If there is a bimodule ${}_SM_T$ with ${}_SM$ faithful and $M_T$ finitely
generated, then $S \preceq T$.
\item
If $T \subseteq S$ and
$S$ is finitely generated as left or right $T$-module, then $S\equiv T$.
\end{enumerate}
\end{lem}

\begin{proof}
(a)
Fix generators $(m_j)_1^r$ of $M_T$ and consider the map
$\mu \colon S \to M^r$, $s \mapsto (sm_j)$; this is a $\k$-linear embedding
as ${}_SM$ is faithful. Let $\sS \subseteq S$ be a finite subset and, for $s \in \sS$, write
$sm_j = \sum_i m_i t_{ij}(s)$ with $t_{ij}(s) \in T$. With  
$\sT = \{ t_{ij}(s) \mid s \in \sS \text{ and all } i,j \}$ and
$\sM = \sum_j \k m_j$, we obtain
$\mu(\sS^{(n)}) \subseteq (\sM\sT^{(n)})^r$ for all $n$.
Therefore, $\dim_\k \sS^{(n)} \le r^2 \dim_\k \sT^{(n)}$, verifying the condition in \eqref{E:Relation}.

\begin{comment}
Assume that $T$ is generated, say as right $S$-module, by the finite set  $\sX$ 
with $1 \in \sX$. We only need to prove $T \preceq S$. So fix a
finite subset $\sT \subseteq T$.
For each $t \in \sT$ and $x \in \sX$, write $tx = \sum_{y \in \sX} y s_{txy}$ 
with  $s_{txy} \in S$ and put $\sS = \{ s_{txy} \mid \text{ all } t,x,y \}$. A straightforward 
induction shows that $\sT^{(n)} \subseteq \sX \sS^{(n)}$ for all $n$. Thus,
$\dim_\k \sT^{(n)} \le |\sX| \dim_\k \sS^{(n)}$, verifying the condition in \eqref{E:Relation}.
\end{comment}

\smallskip

(b) 
Since $T \subseteq S$, we certainly have $T \preceq S$. 
For $S \preceq T$, take $M=S$ in (a).
\end{proof}

%%%%%%%%%%%%%%%%%%%%%%%%%%%%%%%%%%%%%%%%%%%%%%%%%%

\subsection{Subcentralizing extensions}
\label{SS:Sub}

An extension $S \subseteq U$ of $\k$-algebras is called
\emph{centralizing} if $U = C_U(S)S$, where 
$C_U(S) = \{ u \in U \mid su=us \text{ for all } s \in S \}$ is the
centralizer of $S$ in $U$. 
Any subextension $S \subseteq T$ with $T \subseteq U$
will then be called \emph{subcentralizing}.%
\footnote{Finite centralizing extensions and their subextensions have also been called \emph{liberal} 
and \emph{intermediate} extensions, respectively \cite{RS81}.}
Given subsets $\sX_i \subseteq U$, we let $S[\sX_i \mid \text{all } i]$ denote the
$\k$-subalgebra of $U$ that is generated by $S$ and $\bigcup_i \sX_i$\,.

\begin{lem}
\label{L:Sub}
Let $S \subseteq T \subseteq U$ be $\k$-algebras with
$S \subseteq U$ centralizing and assume that $T = S[\sX]$.
Each of the following implies $S\equiv T$:
\begin{enumerate}
\item
$c\sX \subseteq S$ and $(c-d)\sX = 0$ for some 
$d \in S$, $c \in C_U(S)$ with $c$ regular in $U$;
\item
$(\sX S)^m = 0$ for some positive integer $m$;
\item
$\dim_\k \k[\sX] < \infty$ and $S = \k[\sN, \sS]$ 
with $(\sN T)^m = 0$ and $sx=xs$ for all $s \in \sS$, $x \in \sX$.
\end{enumerate}
\end{lem}

\begin{proof}
Since $S \subseteq T$, only $T \preceq S$ needs to be proved in each case. For a given
finite $\sT \subseteq T$, choose finite subsets $\sS' \subseteq S$,
$\sX' \subseteq \sX$ such that $\sT \subseteq \k[\sS', \sX']$.
As we have remarked in \S\ref{SS:Relation},
we may assume that $\sT = \sS' \cup \sX'$. Then $\sT^{(n)}$ is
generated by products of the form
\begin{equation}
\label{E:tau}
\tau = s_0 x_{1} s_1 x_{2} \cdots s_{f-1} x_{f} s_f
\end{equation}
with $x_i \in \sX'$, $s_i$ a product with $\ell_i$ factors from $\sS'$, 
and $0 \le \sum_i \ell_i \le n-f \le n$.  
\smallskip

(a)
Note that $\sT^{(n)} = \sS'^{(n)} \cup \sT_n$\,, where $\sT_n$
denotes the subspace that is generated by all products $\tau$ as in \eqref{E:tau} with $f \ge 1$.
It suffices to show that $\dim_\k \sT_n \le \dim_\k \sS^{(2n)}$ for some finite subset $\sS \subseteq S$.
So let $\tau \in \sT_n$ and put $\sS = \sS' \cup c\sX' \cup \{ d \}$. Then
\[
\begin{aligned}
c^n\tau &=  s_0 c^{n-f}(cx_1) s_1 (cx_2)\cdots s_{f-1}(cx_f)s_f \\
&= s_0 d^{n-f}(cx_1) s_1 (cx_2)\cdots s_{f-1}(cx_f)s_f 
\in \sS^{(2n-f)}.
\end{aligned}
\]
Therefore, $c^n\sT_n \subseteq \sS^{(2n-f)} \subseteq \sS^{(2n)}$.
Since $c$ is regular, it follows that 
$\dim_\k \sT_n \le \dim_\k \sS^{(2n)}$, as desired.
\smallskip

(b) 
Our hypothesis allows us to assume $f < m$ in \eqref{E:tau}. 
For each $x \in \sX'$, write $x = \sum_{i=1}^l c_{xi}y_i$ with $c_{xi} \in C_U(S)$
and $y_i \in S$.
Put $\sS = \sS' \cup \{ y_1, \dots, y_l \}$. Then we can express $\tau$ as a
sum of terms of the form $\gamma\sigma$ with 
$\gamma = c_{x_1i_1}\cdots c_{x_fi_f}$ and 
$\sigma = s_0 y_{i_1} s_1 y_{i_2} \cdots s_{f-1} y_{i_f} s_f \in \sS^{(n)}$.
Thus, $\sT^{(n)} \subseteq \sC \sS^{(n)}$ with $\sC = \{ \text{all } \gamma \}$, a finite subset of $U$,
and so $\dim_\k \sT_n \le |\sC| \dim_\k \sS^{(n)}$. 
\smallskip

(c)
By Lemma~\ref{L:Relation1} it suffices to show that $T$ is finite over $S$. To this end, 
fix a $\k$-basis $(x_j)$ of $\k[\sX]$ and put $\sT = \sS \cup \{ x_j\mid \text{ all } j \}$. All elements
of $T$ are linear combinations of products 
\[
\tau = t_0 n_{1} t_1 n_{2} \cdots t_{f-1} n_{f} t_f
\]
with $n_i \in \sN$, $t_i \in \k[\sT]$ and $f < m$. Each $t_i$ is a linear combination
of products $sx_j$ with $s \in \k[\cS]$. Thus,
we may assume that $\tau$ has the form
\[
\tau = (s_0x_{j_0}) n_{1} (s_1x_{j_1}) n_{2} \cdots (s_{f-1}x_{j_{f-1}}) n_{f} (s_fx_{j_f}).
\]
Finally, we may 
write $x_j = \sum_{i=1}^l c_{ji}y_i$ with $c_{ji} \in C_U(S)$ and $y_i \in S$ as in (b) 
and express $\tau$ as a sum of terms $\gamma \sigma$ with
$\gamma = c_{j_0i_0}\cdots c_{j_fi_f}$ and 
$\sigma = (s_0 y_{i_0})n_1(s_1 y_{i_1})n_2 \cdots n_f(s_{f} y_{i_f}) \in S$.
The (finitely many) products $\gamma$ therefore generate $T$ as $S$-module.
\end{proof}

Lemma~\ref{L:Sub}(a) implies in particular that 
central localizations $S[\sC^{-1}]$, where
$\sC \subseteq \cen S$ is a multiplicative subset consisting of regular elements of $S$,
are growth equivalent to $S$.

%%%%%%%%%%%%%%%%%%%%%%%%%%%%%%%%%%%%%%%%%%%%%%%%%%

\subsection{Characteristic closures}
\label{SS:Closure}

Let $A$ be a $\k$-algebra
that is free of rank $n < \infty$ over a subalgebra $C \subseteq \cen A$. 
For any $a \in A$, let $p_a(t) \in C[t]$ denote the characteristic polynomial of the matrix
$\rho a$, where $\rho$ is the regular representation,
\[
\rho \colon A \into \End_C(A) \cong \Mat_n(C), \quad (\rho a)b = ab \qquad (a,b \in A).
\]
Now let $R$ be an affine $\k$-algebra that is equipped with
an embedding $R \into A$ and let 
\[
T =T_{A/C}(R) \subseteq C
\]
denote
the $\k$-algebra that is generated by the coefficients of all $p_r(t)$ with $r \in R$.
The $\k$-subalgebra of $A$ that is generated by $R$ and $T$ is called
the \emph{characteristic closure}%
\footnote{The characteristic closure is also called the \emph{trace algebra},
especially if $\ch \k = 0$. 
In the special case of a matrix algebra $A = \Mat_d(C)$, we do
of course also have the ordinary 
characteristic polynomial $c_a(t) \in C[t]$ for any $a \in A$.  
The polynomial $p_a(t)$ as defined above
is identical to $c_a(t)^d$  \cite[Exemple 3 on p. ~A III.111]{nB70}. Thus, if
$R \into A$, then $T$
is contained in the algebra that is generated by the coefficients of all
$c_r(t)$ with $r \in R$.} 
 of $R$ over $C$; it is equal to the product $TR \subseteq A$.
Thus,
\[
T \subseteq C \subseteq \cen A\qquad \text{and}\qquad R \subseteq TR \subseteq A.
\]
The following proposition collects some well-known facts about $TR$.

\begin{prop}[notation as above]
\label{P:Closure}
\begin{enumerate}
\item
$TR$ is a finite module over $T$;
\item
$T$ is an affine (commutative) $\k$-algebra;
\item
Assume that $A$ is prime, $A=CR$, and $C = \cen A$. 
Then $t\,TR \subseteq R$ for some $0 \neq t \in T$.
\end{enumerate}
\end{prop}

\begin{proof}
By the Cayley-Hamilton Theorem, $p_r(r) = 0$ for all $r \in R$. So all elements of $R$
are integral over $T$. 
Since $TR$ is $T$-affine, generated by any finite generating set of $R$ as $\k$-algebra, 
it follows that $TR$ is  finite over $T$  \cite[13.8.8]{jMcCjR87}.
This proves (a). Part (b) is now a consequence of the Artin-Tate Lemma;
see \cite[13.9.11]{jMcCjR87}.

For (c), note that $R$ is a prime PI algebra having the same PI degree as $A$, say
$n$ \cite[13.6.7]{jMcCjR87}.
Let  $g_n = g_n(x_1, x_2, \dots)$ be the multilinear polynomial in non-commuting variables $x_i$
as in \cite[13.5.11]{jMcCjR87}. The additive subgroup that is generated by the subset $g_n(R) = 
\{ g_n(r_1, r_2, \dots) \mid r_i \in R \}\subseteq R$ is a nonzero ideal of $T$ \cite[13.9.6]{jMcCjR87}.
Thus, $0 \neq g_n(R)TR \subseteq g_n(R)R \subseteq R$ and (d) holds for any 
$0 \neq t \in g_n(R)$.
\end{proof}

For a general embedding $R \into A$, the connection between $R$ and 
its characteristic closure need not be as
tight as in (c) above as the following simple example shows.

\begin{example}
\label{EX:big} 
Let $R = \k[x]$ and $C = \k[x_1,\dots, x_m]$ be the polynomial algebras and put
$A = C^{m}$. With $C$ embedded diagonally,
$A$ is free over $C$ with basis $(\d_{ij})_{j=1}^m$ $(i = 1,\dots,m)$.
Consider the embedding 
$R \into A$, $r(x) \mapsto r:= (r(x_1),\dots,r(x_m))$. 
Under the regular representation $\rho \colon A \into
\End_C(A) \cong \Mat_m(C)$, we have $\rho r = \diag(r(x_1), \dots, r(x_m))$. 
The characteristic polynomial
$p_r(t)$ is invariant under the action $\Sy_m \acts C$ by permuting the
variables; so its coefficients belong to the algebra $\k[e_1,\dots,e_m]$,
where $e_j$ denotes the $j\Th$ elementary symmetric polynomial. For $r(x) = x$ in particular,
the coefficients are $e_1,\dots, e_m$ up to $\pm$. Thus, $T = \k[e_1,\dots,e_m]$ and so
$\GKdim TR = \GKdim T = m$ while $\GKdim R = 1$. 
\end{example}

%%%%%%%%%%%%%%%%%%%%%%%%%%%%%%%%%%%%%%%%%%%%%%%%%%

\section{Proof of Proposition~\ref{P:Coefficient}}
\label{S:Coeff}

%%%%%%%%%%%%%%%%%%%%%%%%%%%%%%%%%%%%%%%%%%%%%%%%%%

Let $R$ be an affine representable $\k$-algebra and fix an embedding
$R \into \Mat_d(C)$
for some commutative coefficient $\k$-algebra $C$. We may 
choose $C$ to be affine and so noetherian. 
A straighforward noetherian induction shows that there are
finitely many ideals $I_j$ of $C$ such that $\bigcap I_j = 0$ and any two nonzero ideals of 
each $C_j = C/I_j$ have nonzero intersection. It will suffice to produce embeddings
$C_j \into \Mat_{c_j}(P_j)$ with $P_j$ a polynomial algebra over $\k$. Choosing
a large enough polynomial algebra $P$ so as to contain a 
copy of each $P_j$ we then obtain the embedding
\[
C \into \prod C_j \into \prod \Mat_{c_j}(P_j) \into \Mat_c(P),
\]
where  $c = \sum c_j$ and the last map comes from $\Mat_{c_j}(P_j) \into \Mat_{c_j}(P)$ and
lining all $\Mat_{c_j}(P)$ up as blocks along the diagonal of $\Mat_c(P)$.
The original embedding $R \into \Mat_d(C)$ now yields $R \into \Mat_{cd}(P)$ as desired. 

It remains to treat the case where $R$ is affine commutative and any two nonzero ideals of 
$R$ have nonzero intersection. Then any non-nilpotent $r \in R$ is in fact regular: 
$\ann_R(r^l) = \ann_R(r^{l+1}) = \dots$ 
for some $l$ and so $\ann_R(r^l) \cap r^lR = 0$ giving $\ann_R(r^l) = 0$. 
By Noether normalization, $R$ is finite over some polynomial subalgebra 
$P' \subseteq R$. All $0 \neq f \in P'$ are regular in $R$ by the foregoing; so $R \into
R_f = R[f^{-1}]$. Moreover, by generic freeness \cite[Theorem 14.4]{dE95}, we may choose $f$ so that
$R_f$ is free (of finite rank) over $P'_f = P'[f^{-1}]$. The
algebra $P'_f$ is affine Cohen-Macaulay \cite[Theorems 2.1.3(b) and 2.1.9]{wBjH98}.
Therefore, $P'_f$ is a finite-rank free module over some polynomial subalgebra $P\subseteq P'_f$
\cite[Theorem 8.4.2]{mL05} and so
$R_f$ is free of finite rank over $P$ as well.
The regular representation embeds $R_f$ into a matrix algebra over $P$ 
and hence $R$ embeds so as well. This completes the proof of Proposition~\ref{P:Coefficient}.
\qed

\bigskip

The foregoing also shows that, starting with any affine coefficient algebra $C$, the polynomial algebra
$P=\k[x_1,\dots,x_t]$ with $t = \GKdim C$ will work; so $P \equiv C$.

%%%%%%%%%%%%%%%%%%%%%%%%%%%%%%%%%%%%%%%%%%%%%%%%%%

\section{Proof of Theorem~\ref{T:Markov}}
\label{S:Markov}

%%%%%%%%%%%%%%%%%%%%%%%%%%%%%%%%%%%%%%%%%%%%%%%%%%

Recall that every affine representable $\k$-algebra has a coefficient algebra that is a field,
which may be chosen algebraically closed.
Fix an embedding $R \into \Mat_d(K)$ with $K$ an algebraically closed
$\k$-field. The desired commutative algebra $D \equiv R$ will be constructed
as a $\k$-subalgebra of $K$. To this end, identify $R$ with its image in $\Mat_d(K)$
and $K$ with the center of $\Mat_d(K)$ and let $A$ denote
the $\k$-subalgebra of $\Mat_d(K)$ that is generated by $R$ and $K$. So $A$ is 
a finite-dimensional $K$-algebra and
\[
R \subseteq A = KR.
\]
In several steps, we will replace $R$ by growth-equivalent but successively better-conditioned 
affine $\k$-subalgebras $R_i \subseteq A$ before arriving at $D$.
Note that, for any $\k$-algebras $S$, $S'$ with $R \subseteq S \subseteq S' \subseteq A$, 
the extension $S \subseteq A$ is centralizing and so $S \subseteq S'$ is subcentralizing.
We denote the Jacobson radical of $A$ by $\rad A$ and its nilpotence degree by $d$.
\medskip

\textbf{Step 1.} \emph{$R_1 = \bar R \oplus N$, where $N$ is an ideal with $KN = \rad A$
and $\bar R$ is a $\k$-subalgebra that is isomorphic to a finite direct product of 
$\k$-subalgebras $R_e$ with all $KR_e$ simple.}
\smallskip

By the Wedderburn Principal Theorem (e.g., \cite[11.6]{rP82}),
$A = \bar A \oplus \rad A$
for some split semisimple $K$-subalgebra of $\bar A \subseteq A$.
Write each $a \in A$ as 
$a = \bar a + a_{\rad{}}$ according to this decomposition.
Fix a finite set of algebra generators $\sR \subseteq R$ and put
$\bar{\sR} = \{ \bar r \mid r \in \sR\}$ and
$\sN = \{ r_{\rad{}} \mid r \in \sR \}$. Consider the subalgebra
$R':= \k[\bar{\sR},\sN] = R[ \sN] \subseteq A$.
Since $\sN  \subseteq \rad A$, Lemma~\ref{L:Sub}(b) gives $R' \equiv R$. 
Let $\sE$ denote the set of central primitive idempotents of $\bar A$ and
put 
\[
R_1:= R'[\sE] = R[\sN,\sE].
\]
Lemma~\ref{L:Sub}(c), with $S = R_*$\,, $\sS = \bar{\sR}$ and $\sX = \sE$, 
gives $R_1 \equiv R'$\,. Thus, $R_1 \equiv R$.
Further, $R_1 = \bar R \oplus N$, where $N$ 
denotes the ideal that is generated by $\sN$ and
$\bar R = \k[\bar{\sR}, \sE] = \bigoplus_{e \in \sE} R_e$ 
with $R_e = e \k[\bar{\sR}]$; so $\bar R$
is isomorphic to the direct product of the algebras $R_e$.
Since $R \subseteq R_1 \subseteq A = KR$, it follows that 
$KR_1 = A$. So $KN = \rad A$ and $K \bar R = \bar A$.
Therefore, the various $K R_e$ are the simple components $A_e = e\bar A$ of $\bar A$.

\medskip

\textbf{Step 2.}
\emph{$R_2 = \bar R \oplus N$ as for $R_1$; in addition, $\bar R$ is virtually commutative.}
\smallskip

Continuing with the notation of Step 1, we
write $K_e = e K$ $(e \in \sE)$, the center of the simple component $A_e = 
e\bar A = KR_e$ of $\bar A$. 
Let $C_e$ denote the characteristic closure of $R_e$
in $A_e$\,; so $C_e = T_eR_e$ with $T_e:= T_{A_e/K_e}(R_e) \subseteq K_e$ 
as in \S\ref{SS:Closure}. Consider the following $\k$-subalgebras of $\bar A$:
\[
\begin{aligned}
T:= \bigoplus_{e \in \sE} T_e
\qquad\text{and}\qquad
R_*:= \bigoplus_{e \in \sE} C_e = \bar RT.
\end{aligned}
\]
Thus, $T \subseteq \cen\bar A$ and $KR_* = \bar A$. Moreover, $T$ is $\k$-affine 
and $R_*$ is finite over $T$, because this holds for all $T_e \subseteq C_e$
by Proposition~\ref{P:Closure}. So $R_*$ is virtually commutative.
Put 
\[
R_2:= R_*[N] = R_1[T]. 
\]
Since $R_1$ and $T$ are affine, $R_2$ is so
as well. Further, $R_2 = R_* \oplus M$, where $M$ denotes the ideal
of $R_2$ that is generated by $N$; so $KM = \rad A$. By
Proposition~\ref{P:Closure}(c), $t_e T_e \subseteq R_e$ for some $0 \neq t_e \in T_e$. 
Writing $t_e = ek_e$ with $0 \neq k_e \in K$,
we have $k_eT_e = k_eeT_e = t_eT_e \subseteq R_e \subseteq R$.
Applying  Lemma~\ref{L:Sub}(a) with $\sX = T_e$\,, $c = k_e$ and $d = t_e$
for the various $e \in \sE$, 
we obtain $R_1 \equiv R_1[T_e] \equiv R_1[T_e,T_{e'}] \equiv \dots \equiv R_2$.
This completes Step 2. We will continue below with the notation
employed in the statement of Step 1: $\bar R = R_*$\,, $N = M$ and $R_e = C_e$\,.

\medskip

\textbf{Step 3.}
\emph{$R_3 = \bar R \oplus N$ as for $R_2$; in addition, $\bar R$ is commutative.}
\smallskip

Consider the algebra $R_2 = \bar R \oplus N$ in Step 2, with
$\bar R = \bigoplus_{e \in \sE} R_e$ being
finite over $Z:= \cen\bar R = \bigoplus_{e \in \sE} \cen(R_e)$.
Here $\cen(R_e) \subseteq \cen(A_e) = Ke$; so $\cen(R_e) = Z_e e$ 
with $Z_e \subseteq K$. By the Artin-Tate Lemma, $Z$ is affine.
Choose a finite set of algebra generators $\sZ \subseteq Z$ and
a finite set $\sG \subseteq \bar R$ of $Z$-module generators. 
Then $R_2 = \k[\sZ,\sG, \sN]$ for some finite set $\sN \subseteq N$.
Put 
\[
R_3:= Z \oplus N \subseteq R_2\,.
\]
Then $\sG$
generates $R_2$ as (left and right) module over the subalgebra $R_3$. Hence,
$R_3 \equiv R_2$ by Lemma~\ref{L:Relation1}(b). Furthermore, 
$R_3 = \k[\sZ, \sN\sG, \sG\sN]$ is affine.

\medskip

\textbf{Conclusion.}
\emph{The algebra $D$.}
\smallskip

Let $R_3 = Z \oplus N$ be as in Step 3 and choose a finite generating set 
$\sR = \sZ \cup \sN$ for $R_3$
with $\sN \subseteq N$ and $\sZ = \bigcup_{e \in \sE} \sZ_e e\subseteq Z$, 
where each $\sZ_e$ is a set of algebra generators of $Z_e$ with $0, 1 \in \sZ_e$\,. 
The subspace $\sR^{(n)}$ is generated by monomials of the form
\[
\tau = z_0 n_{1} z_1 n_{2} \dots z_{f-1} n_{f} z_f\,,
\]
where $n_j \in \sN$, each $z_k$ is a product with $\ell_k$ factors from $\sZ$,
and $0 \le \sum_k \ell_k \le n-f \le n$. If $\tau \neq 0$, then $f < d$, the
nilpotence degree of $\rad A$, 
and no $z_k$ involves factors from $\sZ_{e}e$ for different $e \in \sE$, 
because these idempotents are orthogonal.
So either $z_k = 1$ or
$z_k = z'_k e_{k}$ with $e_k \in \sE$ and
$z'_k$ a product of $\ell_k$ factors from $\sZ_{e_k}$\,. 
Pulling these $K$-factors out to the left, we can write $\tau = \tau' \mu$ with 
$\tau' = \prod_{z_k \neq 1} z'_k \in K$ and $\mu$ a product
with factors from $\sN \cup \sE$. By the foregoing, there a finitely many such products: $0$
and products with fewer than
$d$ factors from $\sN$ and no adjacent $\sE$-factors. Let $\sM^*$ denote the (finite) set
of words $\mu^*$ in the alphabet $\sN \cup \sE$ such that $\mu^*$ has nonzero image
$\mu \in R_3$ and is reduced using the rule $e^2 = e$ for $e \in \sE$,
%fact that $\sE$ consists of orthogonal idempotents, 
and let $\sZ_{\mu^*}$ be the union of all $\sZ_{e}$ with $e$ occurring in $\mu^*$.
Then $\tau' \in \sZ_{\mu^*}^{(n-f)} \subseteq \sZ_{\mu^*}^{(n)}$ and so
\begin{equation}
\label{E:Markov1}
\sR^{(n)} \subseteq \sum_{\mu^* \in \sM^*} \sZ_{\mu^*}^{(n)} \mu\,.
\end{equation}
On the other hand, $\mu \in \sR^{(2d)}$, since each word $\mu^*$ has length less than $2d$
and $\sN \cup \sE \subseteq \sR$. 
Furthermore, if $e \in \sE$ occurs in $\mu^*$, say $\mu^* = \nu^*e\pi^*$ 
with subwords $\nu^*$ and $\pi^*$, then
$\sZ_e^{(l)}\mu = \nu(\sZ_e e)^{(l)}\pi \subseteq \sR^{(l+2d)}$. It follows that, for all $\mu^*$,
\begin{equation}
\label{E:Markov2}
\sZ_{\mu^*}^{(n)} \mu \subseteq \sR^{(n+2d)} .
\end{equation}
Put $m =|\sM^*|$ and $D^* = \prod_{\mu^*} D_{\mu^*}
= \k[\sD] \subseteq K^{m}$, where $D_{\mu^*} = \k[\sZ_{\mu^*}]$
and $\sD = \prod_{\mu^*} \sZ_{\mu^*}$.
Then $\sD^{(n)} = \prod_{\mu^*} \sZ_{\mu^*}^{(n)}$ and so
$\dim_\k \sR^{(n)} \le \dim_\k \sD^{(n)}  \le m \dim_\k \sR^{(n+2d)}$ 
by \eqref{E:Markov1} and \eqref{E:Markov2}.
Therefore, $R_3  \equiv D^*$. Finally, $\GKdim D^* = \max_{\mu^*} \GKdim  D_{\mu^*}$
\cite[3.1]{wBhK76}
and hence $D^* \equiv D_{\mu^*}$ for some $\mu^*$; see \S\ref{SS:Relation}.
So we may take $D = D_{\mu^*}$ to finish the proof of Theorem~\ref{T:Markov}. 
\qed

\bigskip

I don't know if the algebra $D$ in Theorem~\ref{T:Markov} can be chosen to also be a
coefficient algebra for $R$. An algebra closely related to the specific algebra $D$
constructed in the proof above is indeed a 
coefficient algebra. In detail, continuing with the above notation, recall that 
$R \subseteq R_2 = \sG R_3$ for some finite set $\sG$.
Let $C \subseteq K$ denote the $\k$-algebra that is generated by all $\sZ_e$ with $e \in \sE$ and
let $\sM \subseteq R_3$ be the (finite) multiplicative monoid that is
generated by $\sN \cup \sE$.
It follows from \eqref{E:Markov1} that $R_3 \subseteq \sM C$ and so $R_3\sM \subseteq \sM C$.
Consider the bimodule 
\[
M = R_2 \sM C = \sG R_3 \sM C = \sG\sM C \subseteq A.
\]
Observe that $M$ is finitely generated and torsionfree as $C$-module and faithful as
$R_2$-module, because $R_2 \subseteq M$. Letting $F \subseteq K$ denote the
field of fractions of $C$, we have $F \equiv C$ by Lemma~\ref{L:Sub}(a) and
$M \subseteq M \otimes_C F \cong F^m$ for some $m$. Therefore,
\[
R \subseteq R_2 \into \End_C(M) \into \End_F(M \otimes_C F) \cong \Mat_m(F).
\]
Thus, $F$ is a coefficient algebra for $R$ and some affine subalgebra
of $F$ will work as well.

%%%%%%%%%%%%%%%%%%%%%%%%%%%%%%%%%%%%%%%%%%%%%%%%%%
%  Bibliography
%%%%%%%%%%%%%%%%%%%%%%%%%%%%%%%%%%%%%%%%%%%%%%%%%%

\bibliographystyle{amsplain}
\def\cprime{$'$}
\providecommand{\bysame}{\leavevmode\hbox to3em{\hrulefill}\thinspace}
\providecommand{\MR}{\relax\ifhmode\unskip\space\fi MR }
\providecommand{\MRhref}[2]{%
  \href{http://www.ams.org/mathscinet-getitem?mr=#1}{#2}
}
\providecommand{\href}[2]{#2}

%%%%%%%%%%%%%%%%%%%%%%%%%%%%%%%%%%%%%%%%%%%%%%%%%%

\end{document}